\documentclass[a4paper]{amsart}
\usepackage{amssymb,amscd,amsmath}

\usepackage{tikz}
\usetikzlibrary{matrix}

\theoremstyle{definition}
\newtheorem{formula}{*}[section]
\newtheorem{definition}[formula]{Definition}
\newtheorem{corollary}[formula]{Corollary}
\newtheorem{remark}[formula]{Remark}
\newtheorem{lemma}[formula]{Lemma}
\newtheorem{theorem}[formula]{Theorem}
\newtheorem{example}[formula]{Example}

\newcommand{\forme}[1]{}

\begin{document}

\title{Every 4-equivalenced association scheme is Frobenius.}

\author{Bora Moon}
\address{Department of Mathematics, Postech, 77 Cheongam-Ro. Nam-Gu. Pohang. Gyeongbuk, Korea.}

\email{mbr918@postech.ac.kr}

\date{\today}

\maketitle

\begin{abstract}

For a positive integer $k$, we say that an association scheme $(\Omega,S)$ is {\it $k$-equivalenced} if each non-diagonal element of $S$ has valency $k$. {\it "$k$-equivalenced"} is weaker than pseudocyclic. It is known that \cite{HK}, \cite{MKP}, \cite{MP08} every $k$-equivalenced association scheme is Frobenius when $k$=2,3 and 
\cite{P15} every 4-equivalenced association scheme is pseudocyclic. 
In this paper, we will show that every 4-equivalenced association scheme is Frobenius.
\end{abstract}

\section{Introduction}

For a group $G$ acting transitively on a finite set $\Omega$, $G$ acts on $\Omega \times \Omega$ by $(\alpha, \beta)^{\sigma} = (\alpha^\sigma,\beta^{\sigma})$ where $\alpha, \beta \in \Omega,$ $\sigma \in G$. Let $S_G$ be the set of orbitals of $G$ on $\Omega \times \Omega$. A scheme $(\Omega, S)$ is called {\it schurian} if $S=S_G$ for some permutation group $G$ on $\Omega$. We say a scheme $(\Omega, S)$ is {\it Frobenius} if $S=S_G$ for some Frobenius permutation group on $\Omega$. Let $\mathcal{F}$ be the set of positive integers $k$ such that every $k$-equivalenced association is Frobenius. As we can see in \cite{HK}, \cite{MKP} and  \cite{MP08},  $7,9,11,12,13,14, 15\not \in \mathcal{F}$ and $2,3 \in \mathcal{F}$. In this paper, we will show $4 \in \mathcal{F}$ and some properties of a 4-equivalenced association scheme by using the concept of a plane. As in \cite{P15}, it is known that a 4-equivalenced association scheme $(\Omega,S)$ is Frobenius if $|S| \ge 165$, but it's a very rough bound since $|\Omega|\ge 657$. 
We have $4\in \mathcal{F}$ and \cite{T01} when a 4-equivalenced association scheme exists.


\section{Preliminaries}
\textbf{Introduction of an association scheme}
\begin{definition}
Let $\Omega$ be a finite set and $S$ a partition of $\Omega\times\Omega$ with the following properties: 
\begin{enumerate}
\item $1_{\Omega}=\{(x,x)|x\in\Omega\}\in S$;
\item for $s\in S$, there exists $s^*=\{(y,x)|(x,y)\in s\}\in S$;
\item for $s,t,r\in S, c_{st}^{r}:=|\{z\in\Omega|(x,z)\in s, (z,y)\in t\}|$ is constant for any choice of $(x,y)\in r.$\\

\end{enumerate}

Then, the pair $(\Omega,S)$ is an {\it association scheme} on $\Omega$ and $c_{st}^r$ is called the {\it intersection number} of $(\Omega,S)$.

If for any $s,t,r\in S$, $c_{st}^r=c_{ts}^r$, then $(\Omega,S)$ is {\it commutative}. If for any $s\in S$, $s=s^*$, then $(\Omega,S)$ is {\it symmetric}. For any $s\in S$, $n_s:=c_{ss^*}^{1_{\Omega}}$ is called the {\it valency} of $s$. And, if $n_s=k$ for any $s\in S^{\#}:=S$ $\backslash \{1_\Omega\}$, then $(\Omega,S)$ is called {\it k-equivalenced}.

\end{definition}

For  subsets $R, T \subset S$, $RT = \{ u\in S$ $|$ $c_{r t}^{u} \not = 0$ for some $r\in R, t\in T \}$. We will use the notation $uv :=\{u\}\{v\}=\{w\in S $ $|$ $c_{uv}^w\not =0\}$ for $u, v\in S.$

For a non-empty set $R\subset S$, $R^*:=\{r^*|r\in R\}$.  If $R^*R\subset R$, then we say that $R$ is {\it closed}.\\ 

Let $\langle R\rangle$ be the intersection of all closed subsets of $S$ which contain $R$. For $s\in S$, $\langle s \rangle:=\langle \{s\} \rangle$.

For $s\in S$, a matrix $A_s$ with \[(A_s)_{\alpha\beta}:=\left\{
\begin{array}{ll}
1  & \text{if}~(\alpha,\beta)\in s\\
0  & \text{if}~(\alpha,\beta)\notin s
\end{array} \right.\]

is called the {\it adjacency matrix} of $s$ where the rows and columns of $A_s$ are indexed by $\Omega$.

From the definition of association scheme, $A_uA_v=\sum_{w\in S} c_{uv}^{w}A_w$.

For $A,B\in$ Mat$_{\Omega}(\mathbb{C})$, $\langle A,B\rangle=\frac{1}{|\Omega|}$$Tr(AB^{*})$ : the {\it Hermitian product} where $B^{*}$ is the conjugate transpose of $B$.\\

For an association scheme $(\Omega,S),$ the automorphism group is $Aut(\Omega,S)=\{\sigma\in Sym(\Omega)|\mathrm{r}(\alpha,\beta)=\mathrm{r}(\alpha^{\sigma},\beta^{\sigma})$ for all $\alpha, \beta\in \Omega$ $\}$.

An association scheme $(\Omega, S)$ is called {\it Frobenius} if for some Frobenius group G which acts on $\Omega$, $(\Omega, S)=(\Omega, S^G)$ where $S^G$ is the set of orbitals of G on $\Omega \times \Omega$.\\

\begin{theorem}\cite{ZYM03}\label{std}
Let $(\Omega,S)$ be a $k$-equivalenced association scheme. If $k$ is even, then $(\Omega,S)$ is symmetric.
  
\end{theorem}

\begin{theorem}\cite{MKP},\cite{MP08}\label{std}
Every $k$-equivalenced association scheme is Frobenius when $k=2,3$.
\end{theorem}

The {\it indistinguishing number} of a relation $s \in S$ is defined by the number
 $c(s)=\sum_{u\in S^{\#}}c_{uu^{*}}^s$. The maximum of $c(s)$, $s \in S^{\#}$, is called the {\it indistinguishing number} of an association scheme $(\Omega, S)$.\\

We replace the definition of \textit{pseudocyclic} with the following theorem.

\begin{theorem} \cite{MP12}\label{theorem2.3} Let $(\Omega,S)$ be an association scheme. The following statements are equivalent.
\begin{enumerate}
\item $(\Omega, S)$ is a pseudocyclic association scheme with valency $k$.
\item $(\Omega, S)$ is a $k$-equivalenced scheme with $c(s) = k-1$ for all $s \in S^{\#}$.
\end{enumerate}

\end{theorem}

\begin{theorem}\cite{P15}\label{std}
Every 4-equivalenced association scheme is pseudocyclic and transitive.
\end{theorem}

\textbf{Notation}

In this paper, we use $u$ (resp. $u\cdot v$) instead of $A_u$(resp. $A_u A_v$). \it{i.e.} $u\cdot v=\sum_{w\in S}c_{uv}^{w}w$ but $uv$=$\{s \in S | c_{uv}^s \ne 0 \}$. And $\mathrm{r}(\alpha,\beta)$ is a unique relation which contains $(\alpha, \beta)$ and $S^{\#}=S$ $\backslash$ $\{1_\Omega\}$.\\

\textbf{Elementary properties of a 4-equivalenced association scheme}\\

From now on, we assume that $(\Omega, S)$ is a 4-equivalenced association scheme. The following results are proved in \cite{P15}.

\begin{lemma}\cite{P15}\label{std}
For any $s\in S^{\#}$, one of the following holds.
\begin{enumerate}
\item $s\cdot s=4\cdot 1_{\Omega}+2u+v $ for some distinct $u,v\in S^{\#}$
\item $s\cdot s=4\cdot 1_{\Omega}+3s$
\end{enumerate}
\end{lemma}

Define $S_i=\{s\in S||ss|=i\}$ for $i=1,2,3$. We can define two mappings $\varphi$ and $\psi$ from $S_3$ to $S^{\#}$ such that $s\cdot s=4\cdot 1_{\Omega}+2s^{\varphi}+s^{\psi}$.\\

After some elementary computations, we can modify \cite[Lemma 3.1]{P15} as follow:

\begin{lemma}\label{lemma2.6}
For $s,t\in S^{\#}$ with $s\ne t$, 
\begin{enumerate}
\item $s\cdot t=u_1+u_2+u_3+u_4$ for some distinct $u_i$'s \textit{iff}  $s\ne t^{\varphi}$ and $t\ne s^{\varphi}$
\item $s\cdot t=u_1+u_2+2s$ for some distinct $u_i$'s \textit{iff} $s\ne t^{\varphi}$ and $t=s^{\varphi}$
\item $s\cdot t=2s+2t$ \textit{iff} $s= t^{\varphi}$ and $t= s^{\varphi}$
\end{enumerate}

\end{lemma}

\begin{lemma}\cite{P15}\label{std}
For any $s\in S_3$,  $s^{\varphi}\in S_3$ and $s^{\psi}=(s^{\varphi})^{\varphi}$. Further, $\varphi$ and $\psi$ are bijections on $S_3$.
  
\end{lemma}

Now, we can extend $\varphi$ and $\psi$ to be permutations of $S$ whose restrictions on the complement of $S_3$ is an identity map. \textit{i.e.} for any $s\in S_2$, $s^\varphi=s^\psi=s$.\\

\textbf{Construction of planes}

See \cite{P15} for the definition of a plane $\mathcal{P}$ on a 4-equivalenced association scheme ($\Omega,S$) and some basic properties.

The next definition is an extended version $\widetilde{\mathcal{P}}$ of the plane $\mathcal{P}$. 

\begin{definition}

Let $s\in S^{\#}$ and $\alpha\in \Omega$, $\beta,\gamma,\delta,\epsilon\in\alpha s$ with $\mathrm{r}(\beta,\delta)=\mathrm{r}(\gamma,\epsilon)=s^{\psi}$.

$\bullet$ For $s\in S_2$, we define a map
\[\widetilde{\mathcal{P}}{(\alpha,\beta,\gamma,\delta,\epsilon)_s}:\{(0,0),(1,0),(0,1),(-1,0),(0,-1)\}\rightarrow\Omega\textrm{~by~}(i,j)\mapsto\alpha_{i,j}\]
as follows:  $\alpha_{0,0}:=\alpha$, $\alpha_{1,0}:=\beta$, $\alpha_{0,1}:=\gamma$, $\alpha_{-1,0}:=\delta$, $\alpha_{0,-1}:=\epsilon$.\\

$\bullet$ For $s\in S_3$, we define a map
\[\widetilde{\mathcal{P}}{(\alpha,\beta,\gamma,\delta,\epsilon)_s}:\mathbb{Z}\times\mathbb{Z}\rightarrow\Omega\textrm{~by~}(i,j)\mapsto\alpha_{i,j}\]
as follows: 
\begin{enumerate}

  \item $\alpha_{0,0}:=\alpha$, $\alpha_{1,0}:=\beta$, $\alpha_{0,1}:=\gamma$, $\alpha_{-1,0}:=\delta$, $\alpha_{0,-1}:=\epsilon$.

  \item  Define $\alpha_{i,0}$ and $\alpha_{0,i}$ inductively for $|i|\geq2$.

       \[\alpha_{i,0}:=\begin{cases}\mathcal{P}(\alpha,\beta,\gamma)_s(i,0) \textrm{~if~} i \ge 2,\\
       \mathcal{P}(\alpha,\delta,\epsilon)_s(-i,0)\textrm{~if~}i \le -2.
       \end{cases}\]

      \[ \alpha_{0,i}:=\begin{cases}\mathcal{P}(\alpha,\beta,\gamma)_s(i,0) \textrm{~if~} i\ge 2,\\
               \mathcal{P}(\alpha,\delta,\epsilon)_s(0,-i)\textrm{~if~}i \le -2.
          \end{cases} \]

       \item

$$\alpha_{i,j}:=
\begin{cases}\mathcal{P}(\alpha,\beta,\gamma)_s(i,j) \textrm{~if~} i\ge 1, j\ge 1,\\

(\mathcal{P}(\alpha,\gamma,\delta)_s\circ \sigma)(i,j) \textrm{~if~}  i\le -1, j\ge 1,\\
(\mathcal{P}(\alpha,\delta,\epsilon)_s\circ \sigma^2)(i,j) \textrm{~if~}  i\le -1, j\le -1,\\
(\mathcal{P}(\alpha,\epsilon,\beta)_s\circ \sigma^3)(i,j) \textrm{~if~}  i\ge 1, j\le -1 .        
\end{cases} $$

where a map $\sigma:\mathbb{Z}\times \mathbb{Z}$ $\rightarrow$ $\mathbb{Z}\times \mathbb{Z}$ is defined by $(i,j)$ $\rightarrow$ $(-j,i)$.

\end{enumerate}

We can see that the map $\widetilde{\mathcal{P}}(\alpha,\beta,\gamma,\delta,\epsilon)_s$ is well-defined in \cite{P15}.
We say that the map $\widetilde{\mathcal{P}}(\alpha,\beta,\gamma,\delta,\epsilon)_s$ is an {\it $s$-plane}.\medskip

\end{definition}

For a given $s$-plane $\widetilde{\mathcal{P}}(\alpha,\alpha_1,\alpha_2,\alpha_3,\alpha_4)_s$, 
$\sigma$ acts on the $\widetilde{\mathcal{P}}(\alpha,\alpha_1,\alpha_2,\alpha_3,\alpha_4)_s$ by $\widetilde{\mathcal{P}}(\alpha,\alpha_1,\alpha_2,\alpha_3,\alpha_4)_s\circ \sigma$.\\

Let $\alpha\langle s \rangle$ be the set of all elements of a set $\Omega$ on the $s$-plane which has $\alpha$ as the $(0,0)$-th image.

\begin{theorem}\cite{P15} For $s\in S^{\#}$,  let $\alpha_{i,j}$ be an $(i,j)$-th point of an $s$-plane $\widetilde{\mathcal{P}}(\alpha,\beta,\gamma,\delta,\epsilon)_s$. \textit{i.e.} $\alpha_{i,j}=\widetilde{\mathcal{P}}(\alpha,\beta,\gamma,\delta,\epsilon)_s(i,j)$ where $(i,j)$ is in the domain of the $s$-plane. 

Then, $\mathrm{r}(\alpha,\alpha_{i,j})=\mathrm{r}(\alpha,\alpha_{k,l})$ if $(k,l)\in \{(i,j),(-j,i),(-i,-j),(j,-i)\}$

\end{theorem}


\section{Every $4$-equivalenced association scheme is Frobenius}

\begin{example}
Let $\Omega$ be a set with $|\Omega|=5$ and $S=\{1_\Omega, s\}$ where $s=\Omega \times \Omega -1_\Omega$. It is easy to see that $(\Omega,S)$ is a 4-equivalenced and Frobenius. 
\end{example}

Note that $S=\langle s\rangle$ for some $s\in S$ or there exists the elements $s$ and $t$ of $S$ such that $s \not \in \langle t \rangle$ and $t \not \in \langle s \rangle$. So, for $s,t\in S$, if  $s \not \in \langle t \rangle$ and $t \not \in \langle s \rangle$, we notate it by $s\wr t$.

\begin{lemma}\label{lemma3.7}
For  $s, t\in S$ with $s\wr t$, then $s \cdot t =a_1+a_2+a_3+a_4$ where $a_i$'s are the distinct elements of $S$ and not in $\langle t\rangle \cup \langle s\rangle$.
\end{lemma}

\begin{proof}
Since $\langle s \cdot t, s\cdot t\rangle =\langle s \cdot s, t\cdot t\rangle =16$ and if $a_i \in  \langle t\rangle \cup \langle s\rangle$, then $s \in \langle t\rangle$ or $t \in \langle s\rangle$.\\
\end{proof}

\begin{lemma}\label{lemma3.8}
For $s, t\in S_3$ with $s\wr t$, we have $|t^{\varphi}s^{\varphi}\cap t^{\psi}s^{\psi}|\le 1$.
\end{lemma}

\begin{proof}
If not, $\langle s^{\varphi}\cdot t^{\varphi},s^{\psi}\cdot t^{\psi}\rangle=\langle s^{\varphi}\cdot s^\psi, t^{\varphi}\cdot t^\psi \rangle=\langle 2s^{\varphi}+a_1+a_2 $ or $ 2s^{\varphi}+2s^{\psi}, 2t^{\varphi}+b_1+b_2$ or $ 2t^{\varphi}+2t^{\psi} \rangle \ge 8$ by Lemma \ref{lemma2.6}. It must be $a_1=b_1$ and $a_2=b_2$ or $a_1=b_2 $ and $ a_2=b_1$. But, here $(s^\varphi)^{\psi}\in a_1a_2\subset\langle t \rangle$ and $(t^\varphi)^{\psi}\in b_1b_2\subset\langle s \rangle$ by the relations on the $s^\varphi$-plane and the $t^\varphi$-plane. It contradicts to our assumption.\\
\end{proof}

\begin{definition} For $
\alpha \in \Omega$ and $s,t\in S^{\#}$, $s \sim_\alpha t$ if there exist a pair of $s$-plane and $t$-plane which have $\alpha$ as the $(0,0)$-th image such that all relations on  $\alpha\langle s\rangle$ $\times$ $\alpha\langle t\rangle$ are invariant under the map $\sigma$.


\end{definition}

Choose a point $\alpha \in \Omega$.

\begin{theorem}\label{theorem3.10}
For $s \in S_3,$ $ t\in S^{\#}$ with $s\wr t$, we have $s \sim_\alpha t$

\end{theorem}

\begin{proof}
\textbf{STEP 1)}
By Lemma ~\ref{lemma3.7}, $s\cdot t=a_1+a_2+a_3+a_4$ for some distinct $a_i$'s of $S$ with $a_i\not\in \langle s\rangle \cup\langle t\rangle$. Take $\alpha_1,\alpha_3\in \alpha s$ with $\mathrm{r}(\alpha_1,\alpha_3)=s^{\psi}$ and $\beta_1,\beta_3\in \alpha t$ with $\mathrm{r}(\beta_1,\alpha_1)=\mathrm{r}(\beta_3,\alpha_3)=a_1$. Let $a_3:=\mathrm{r}(\alpha_1,\beta_3)$. \\

Suppose $a_2:=\mathrm{r}(\alpha_3,\beta_1)\ne a_3$. 
For $\beta_2,\beta_4 \in \alpha t$, we may assume that $\mathrm{r}(\alpha_1,\beta_2)=a_2$, $\mathrm{r}(\alpha_1,\beta_4)=a_4$, $\mathrm{r}(\alpha_3,\beta_2)=a_4$, $\mathrm{r}(\alpha_3,\beta_4)=a_3$. 
Take $\alpha_2, \alpha_4\in \alpha s$ with $\mathrm{r}(\beta_2,\alpha_2)=\mathrm{r}(\beta_4,\alpha_4)=a_1$. See the left figure in Figure~\ref{figure1}.
The relations $a_1,a_2,a_3$ and $a_4$ are represented by the red, blue, green and brown lines in Figure \ref{figure1}. Using $s\cdot t=a_1+a_2+a_3+a_4$, we get $\mathrm{r}(\beta_1,\alpha_2)=\mathrm{r}(\beta_2,\alpha_4)=a_3$. Then, $s^{\varphi},s^\psi$ $\in a_1a_3$. Since $\langle s^\varphi \cdot a_1, s^\psi \cdot a_1\rangle \ge 4$, $a_1\in \langle s\rangle$. So, $\mathrm{r}(\alpha_3,\beta_1)=a_3$. \\

For some $\alpha_2,\alpha_4\in \alpha s$ with $\mathrm{r}(\alpha_2,\alpha_4)=s^\psi$, $\mathrm{r}(\beta_2,\alpha_2)=\mathrm{r}(\beta_4,\alpha_4)=a_1$(and then $\mathrm{r}(\beta_2,\alpha_4)=\mathrm{r}(\beta_4,\alpha_2)=a_3$). So, $\mathrm{r}(\beta_1,\alpha_2)$ is $a_2$ or $a_4$. \\
Suppose $\mathrm{r}(\beta_1,\alpha_2)=a_2$ (and then all the relations on $\alpha s\times \alpha t$ are determined by it). See the right figure in Figure~\ref{figure1}.
Now, we will get some properties of relations $s$, $t$ and $a_i$'s depending on the condition of $t$.\\

Note that $s^{\varphi}\in  a_1a_2\cap a_3a_4  \cap a_1a_4\cap a_2a_3$ and $ s^{\psi}\in a_1a_3\cap a_2a_4$ always hold. \\

\begin{figure}[h]

\begin{tikzpicture}
  \matrix (m) [matrix of math nodes,row sep=0.3cm,column sep=0.3cm]
  { &  & \alpha_1 & & & & &  & \alpha_1 & &  \\
 	 & \beta_4 & & \beta_1& & &  & \beta_4 & & \beta_1&  \\
 	\alpha_4 & & \alpha &   & \alpha_2 & & \alpha_4 & & \alpha &   & \alpha_2 \\
	 & \beta_3 &   & \beta_2  & & &  & \beta_3 &   & \beta_2  & \\
	  & & \alpha_3 &   & & &  & & \alpha_3 &   &\\};
  \path[dashed,-]
    (m-3-3) edge (m-1-3)
    (m-3-3) edge (m-5-3)
    (m-3-3) edge (m-3-1)
    (m-3-3) edge (m-3-5)
    ;

  \path[-]
    (m-3-3) edge (m-2-2)
    (m-3-3) edge (m-2-4)
    (m-3-3) edge (m-4-2)
    (m-3-3) edge (m-4-4)
    ;

 \path[-,color=red]
   (m-1-3) edge (m-2-4)
   (m-3-5) edge (m-4-4)
   (m-5-3) edge (m-4-2)
   (m-3-1) edge (m-2-2);

 \path[-,color=blue]
   (m-1-3) edge (m-4-4)

   (m-5-3) edge (m-2-4);

    \path[-,color=green]
   (m-1-3) edge (m-4-2)

   (m-5-3) edge (m-2-2);

    \path[-,color=brown]
   (m-1-3) edge (m-2-2)

   (m-5-3) edge (m-4-4);

     \path[dashed,-]
    (m-3-9) edge (m-1-9)
    (m-3-9) edge (m-5-9)
    (m-3-9) edge (m-3-7)
    (m-3-9) edge (m-3-11)
    ;

  \path[-]
    (m-3-9) edge (m-2-8)
    (m-3-9) edge (m-2-10)
    (m-3-9) edge (m-4-8)
    (m-3-9) edge (m-4-10)
    ;

 \path[-,color=red]
   (m-1-9) edge (m-2-10)
   (m-3-11) edge (m-4-10)
   (m-5-9) edge (m-4-8)
   (m-3-7) edge (m-2-8);

 \path[-,color=blue]
   (m-1-9) edge (m-4-10)
   (m-3-11) edge (m-2-10)
   (m-5-9) edge (m-2-8)
   (m-3-7) edge (m-4-8);
   
    \path[-,color=green]
   (m-1-9) edge (m-4-8)
   (m-3-11) edge (m-2-8)
   (m-5-9) edge (m-2-10)
   (m-3-7) edge (m-4-10);
   
    \path[-,color=brown]
   (m-1-9) edge (m-2-8)
   (m-3-11) edge (m-4-8)
   (m-5-9) edge (m-4-10)
   (m-3-7) edge (m-2-10);
   
\end{tikzpicture} 

\caption{}\label{figure1}
\end{figure}
 
$\bullet$ {\it case 1)}  $t\in S_2$\\
Since $t\cdot a_i=s+\sum_{j\ne i}a_j$,  $\langle s^{\varphi}\cdot a_i,  t\cdot a_i \rangle\ge8, \langle s^{\psi}\cdot a_i, t\cdot a_i\rangle\ge 4 $ 
and $\langle s^{\varphi}\cdot t, s^{\psi} \cdot t \rangle=0$.
 So, $s^{\varphi} \cdot t=\sum_{i=1}^4a_i^{\varphi},$ $s^{\psi} \cdot t=\sum_{i=1}^4a_i^{\psi},$ $a_i\in S_3 $ and there is no $(i,j)$ such that $a_i^{\varphi}=a_j\cdots$(5).\\

$\bullet$ {\it case 2)} $t\in S_3$ 

Note that exactly one of the following must hold.\\

(a) $\mathrm{r}(\beta_1,\beta_2)=\mathrm{r}(\beta_1,\beta_4)=t^{\varphi} \in a_1a_2\cap a_3a_4\cap a_1a_4\cap a_2a_3$ and $\mathrm{r}(\beta_1,\beta_3)= t^{\psi}\in a_1a_3\cap a_2a_4$

(b) $\mathrm{r}(\beta_1,\beta_2)=\mathrm{r}(\beta_1,\beta_3)=t^{\varphi} \in a_1a_2\cap a_3a_4\cap a_1a_3\cap a_2a_4 $ and $\mathrm{r}(\beta_1,\beta_4)=t^{\psi}\in a_1a_4\cap a_2a_3$

(c) $\mathrm{r}(\beta_1,\beta_3)=\mathrm{r}(\beta_1,\beta_4)=t^{\varphi} \in a_1a_3\cap a_2a_4\cap a_1a_4\cap a_2a_3$ and $\mathrm{r}(\beta_1,\beta_2)= t^{\psi}\in a_1a_2\cap a_3a_4$.\\

Suppose $(a)$ holds. \\
Then, for $i=1,2,3,4,$ $8\le \langle a_i\cdot s^{\varphi}, a_i\cdot t^{\varphi}\rangle, 4\le \langle a_i\cdot s^{\psi},a_i\cdot t^{\psi}\rangle$$\cdots (1)$.

If $c_{s^{\varphi}t^{\varphi}}^{a_i^{\psi}}=2$, then $a_i\in \langle s\rangle \cup \langle t\rangle$. By Lemma~\ref{lemma3.7}, $c_{t^{\varphi}s^{\varphi}}^{a_i^{\varphi}}\ge1$ for $i=1,2,3,4$. So, $s^{\varphi}\cdot t^{\varphi}=a_1^{\varphi}+a_2^{\varphi}+a_3^{\varphi}+a_4^{\varphi}\cdots$ (2) since $a_i$'s are distinct. 

By $(1)$, $a_i^{\varphi}$ or $a_i^{\psi}\in t^{\psi}s^{\psi}\cdots (3)$.
By $(2), (3)$ and Lemma ~\ref{lemma3.8}, there is at most one pair $(k,l)$ such that $a_k^{\varphi}=a_l$ and there is at most one $a_i\in S_2$ $\cdots(4)$.\\

Suppose $(b)$ or $ (c)$. \\For $i=1,2,3,4,$ 
$\langle a_i\cdot s^{\varphi},a_i\cdot t^{\varphi}\rangle, \langle a_i\cdot s^{\psi},a_i\cdot t^{\varphi}\rangle, \langle a_i\cdot s^{\varphi},a_i\cdot t^{\psi}\rangle \ge 4.$ So, for $i=1,2,3,4,$ $a_i^{\varphi}$ or $a_i^{\psi} \in s^{\varphi}t^{\varphi}$, $a_i^{\varphi}$ or $a_i^{\psi} \in s^{\varphi}t^{\psi}$ and $a_i^{\varphi}$ or $a_i^{\psi} \in s^{\psi}t^{\varphi}.$ 

Because $\langle s^{\varphi}\cdot t^{\varphi},s^{\psi}\cdot t^{\varphi}\rangle=\langle s^{\varphi}\cdot t^{\varphi},s^{\varphi}\cdot t^{\psi}\rangle=0$, 
 for $i=1,2,3,4$ $a_i^{\varphi^{j_i}}\in s^{\varphi}t^{\psi}\cap s^{\psi}t^{\varphi}$ for some $j_i\in\{1,2\}$ and $a_i\in S_3$.  We get $|s^{\varphi}t^{\psi}\cap s^{\psi}t^{\varphi}|\ge2 $ since $a_i$'s are distinct. But it contradicts to Lemma~\ref{lemma3.8}. So, $(b)$ and $(c)$ cannot happen.\\

\begin{lemma}\label{lemma3.11}
For distinct $i,j,k,l\in$ $\{1,2,3,4\}$, we have $a_i^{\varphi^m}, a_j^{\varphi^n}\in a_ka_l$ 
for some $m,n \in \{1,2\}$ where $i+j \equiv k+l \equiv 0 (mod 2)$. And, for any distinct ${i'}, {j'}\in $$\{1,2,3,4\}$, $s^{\varphi^{m'}},t^{\varphi^{n'}} \in a_{{i'}}a_{{j'}}$ for some ${m'},{n'}\in\{1,2\}.$

\end{lemma}

\begin{proof}

There is $\gamma_1\in \Omega$ such that $\mathrm{r}(\beta_1,\gamma_1)=a_1, \mathrm{r}(\gamma_1,\beta_2)=a_4 $ since $c_{a_1a_4}^{t^{\varphi}} \ge 1  $.\\Then, $\langle a_1\cdot a_1, a_2\cdot a_4\rangle, \langle a_4\cdot a_4,a_1\cdot a_3\rangle>0,$ so, there is $m_1,n_1 \in \{1,2\}$ such that $a_4^{\varphi^{m_1}}\in a_1a_3,$
$a_1^{\varphi^{m_2}}\in a_2a_4$.  Similarly, we can choose $\gamma_2, \gamma_3\in\Omega$ such that $\mathrm{r}(\beta_2,\gamma_2)=a_4, \mathrm{r}(\gamma_2,\beta_3)=a_3,$ $\mathrm{r}(\beta_4,\gamma_3)=a_2$ and $\mathrm{r}(\gamma_3, \beta_1)=a_1$. Then, $a_2^{\varphi^{n_1}}\in a_1a_3$ and 
$a_3^{\varphi^{n_2}}\in a_2a_4$ for some $m_2,n_2\in\{1,2\}.$ 

\end{proof}

\begin{lemma}\label{lemma3.12}
We get $|a_ia_k \cap a_ja_k|\le 2$ if $i+j \equiv 0 (mod 2)$ and $a_k\in S_3$ for distinct $i,j,k\in \{1,2,3,4\}$.
\end{lemma}

\begin{proof}
If $|a_ia_k\cap a_ja_k|\ge3$, then $12 \le \langle a_i\cdot a_k, a_j\cdot a_k\rangle$. Note that $\langle a_i\cdot a_k, a_j\cdot a_k\rangle \le 16$.

$i)$ $\langle a_i\cdot a_k, a_j\cdot a_k\rangle=12$. \\
Then, $c_{a_ia_j}^{a_k^{\varphi}}=c_{a_ia_j}^{a_k^{\psi}}=1$. So, $a_k^{\varphi},a_k^{\psi}\in a_ia_j$. But, by Lemma ~\ref{lemma3.11}, for $l\in\{1,2,3,4\} \backslash \{i,j,k\}$, $a_k^{\varphi^{i_1}}, a_l^{\varphi^{i_2}}\in a_ia_j$ for some $i_1, i_2\in \{1,2\}$. 
 So, $a_l \in S_3$ and $a_l^{\varphi}=a_k$ or $a_k^{\varphi}=a_l$ for $k+l\equiv0$ $(mod 2)$. If $a_l^{\varphi}=a_k$, then $a_k\cdot a_l=2a_l+t^{\varphi^{m}}+s^{\varphi^{n}}$ for some $m,n\in \{1,2\}$ by Lemma \ref{lemma3.11}. Then, by Lemma \ref{lemma3.11}, $a_i=a_j^{\varphi}$ or $a_j=a_i^{\varphi}$. Contradict to (4) and (5). (Similarly, we can do this when $a_k^{\varphi}=a_l$.)\\
 
$ii)$ $\langle a_i\cdot a_k, a_j\cdot a_k\rangle=16$.\\
 Then, $c_{a_ia_j}^{a_k^{\varphi}}=2$ or $c_{a_ia_j}^{a_k^{\psi}}=2 $ and $c_{a_ia_j}^{a_k^{\varphi}}=1$.\\
If $c_{a_ia_j}^{a_k^{\varphi}}=2$, then, $a_i=a_k^{\varphi}$ and $a_j=a_i^{\varphi}$ or $a_j=a_k^{\varphi}$ and $a_i=a_j^{\varphi}$. It cannot happen.\\
If $c_{a_ia_j}^{a_k^{\psi}}=2$ and $c_{a_ia_j}^{a_k^{\varphi}}=1$, then $a_ia_j= *+a_k^{\varphi}+2a_k^{\psi}$. But $s^{\varphi^{m}}, t^{\varphi^n}, \in a_ia_j$ for some $m,n\in\{1,2\}$ by Lemma~\ref{lemma3.11}..\\ 
Because of $a_i\not\in\langle s\rangle\cup \langle t\rangle$ for $i=1,2,3,4,$ we get a contradiction.\\

So, $|a_ia_k\cap a_ja_k|\le 2 $ if $i+j\equiv 0$ $( mod 2),$ $a_k\in S_3$\\

\end{proof}

\begin{remark}
In fact, for $i,j,k,l$ with $\{i,j,k,l\}=\{1,2,3,4\}$, we can get $a_i^{\varphi^m}, a_j^{\varphi^n} \in a_ka_l$ for some $m,n\in\{1,2\}$ when $t\in S_2$. Furthermore, $|a_ia_k \cap a_ja_k|\le 2$ if $i+j \equiv 0 (mod 2)$ for any distinct $i,j,k\in \{1,2,3,4\}$.  
\end{remark}

We will end this proof by using above lemmas, (4) and (5).\\

 At least one of $\{a_1,a_4\}$, $\{a_2,a_3\}$ is contained in $S_3$ whether $t\in S_2$ or not.\\
  If $a_2,a_3\in S_3$, choose $\gamma\in\Omega$ such that $\mathrm{r}(\beta_3,\gamma)=a_3, \mathrm{r}(\gamma,\beta_4)=a_2.$ Then, $w:=\mathrm{r}(\gamma,\alpha_4)\in a_1a_2\cap a_2a_3$. By Lemma~\ref{lemma3.12}, $|a_1a_2\cap a_2a_3|\le 2$ and $s^{\varphi},t^{\varphi}\in a_1a_2\cap a_2a_3$. So, $w=t^{\varphi} $ or $w=s^{\varphi}$. If $w=s^{\varphi},$ $\mathrm{r}(\gamma,\alpha_1)\in a_3a_3\cap s^{\varphi}s^{\varphi}=\emptyset$. If $w=t^{\varphi}$, then $\mathrm{r}(\gamma,\alpha_2)\in t^{\varphi}s^{\psi}\cap a_2a_3\cap a_3a_4\ne \emptyset$. Since $|a_2a_3\cap a_3a_4|\le 2$ and $t^{\varphi}, s^{\varphi}\in a_2a_3\cap a_3a_4$, it contradicts to $s\wr t$.
  By the same reason, if $a_1,a_4\in S_3$, we can use the same arguments.(Choose $\gamma$ such that $\mathrm{r}(\beta_1,\gamma)=a_1,\mathrm{r}(\beta_2,\gamma)=a_4$)\\

So, $\mathrm{r}(\beta_1,\alpha_2)=a_2 $ cannot hold.

Consequently, we get the relations as follows (See the left figure of Figure\ref{figure2}):

$$\begin{cases}
a_1=r(\alpha_1,\beta_1)=r(\alpha_2,\beta_2)=r(\alpha_3,\beta_3)=r(\alpha_4,\beta_4),\\
 a_2=r(\alpha_1,\beta_2)=r(\alpha_2,\beta_3)=r(\alpha_3,\beta_4)=r(\alpha_4,\beta_1), \\
a_3=r(\alpha_1,\beta_3)=r(\alpha_2,\beta_4)=r(\alpha_3,\beta_1)=r(\alpha_4,\beta_2),\\
a_4=r(\alpha_1,\beta_4)=r(\alpha_2,\beta_1)=r(\alpha_3,\beta_2)=r(\alpha_4,\beta_3).

\end{cases}$$

\begin{figure}[h]

\begin{tikzpicture}
  \matrix (m) [matrix of math nodes,row sep=0.3cm,column sep=0.3cm]
  { &  & \alpha_1 & & & & &  & \alpha_1 & &  \\
 	 & \beta_4 & & \beta_1& & &  & \delta_4 & & \delta_1&  \\
 	\alpha_4 & & \alpha &   & \alpha_2 & & \alpha_4 & & \alpha &   & \alpha_2 \\
	 & \beta_3 &   & \beta_2  & & &  & \delta_3 &   & \delta_2  & \\
	  & & \alpha_3 &   & & &  & & \alpha_3 &   &\\};
  \path[dashed,-]
    (m-3-3) edge (m-1-3)
    (m-3-3) edge (m-5-3)
    (m-3-3) edge (m-3-1)
    (m-3-3) edge (m-3-5)
    ;

  \path[-,color=black]
    (m-3-3) edge (m-2-2)
    (m-3-3) edge (m-2-4)
    (m-3-3) edge (m-4-2)
    (m-3-3) edge (m-4-4)
    ;

 \path[-,color=red]
   (m-1-3) edge (m-2-4)
   (m-3-5) edge (m-4-4)
   (m-5-3) edge (m-4-2)
   (m-3-1) edge (m-2-2);

 \path[-,color=blue]
   (m-1-3) edge (m-4-4)
   (m-3-5) edge (m-4-2)
   (m-5-3) edge (m-2-2)
   (m-3-1) edge (m-2-4);
   
    \path[-,color=green]
   (m-1-3) edge (m-4-2)
   (m-3-5) edge (m-2-2)
   (m-5-3) edge (m-2-4)
   (m-3-1) edge (m-4-4);
   
    \path[-,color=brown]
   (m-1-3) edge (m-2-2)
   (m-3-5) edge (m-2-4)
   (m-5-3) edge (m-4-4)
   (m-3-1) edge (m-4-2);

     \path[dashed,-]
    (m-3-9) edge (m-1-9)
    (m-3-9) edge (m-5-9)
    (m-3-9) edge (m-3-7)
    (m-3-9) edge (m-3-11)
    ;

  \path[-]
    (m-3-9) edge (m-2-8)
    (m-3-9) edge (m-2-10)
    (m-3-9) edge (m-4-8)
    (m-3-9) edge (m-4-10)
    ;

 \path[-,dashed,color=red]
   (m-1-9) edge (m-2-10)
   (m-3-11) edge (m-4-10)
   (m-5-9) edge (m-4-8)
   (m-3-7) edge (m-2-8);

 \path[-,dashed,color=blue]
   (m-1-9) edge (m-4-10)
   (m-3-11) edge (m-4-8)
   (m-5-9) edge (m-2-8)
   (m-3-7) edge (m-2-10);
   
    \path[-,dashed,color=green]
   (m-1-9) edge (m-4-8)
   (m-3-11) edge (m-2-8)
   (m-5-9) edge (m-2-10)
   (m-3-7) edge (m-4-10);
   
    \path[-,dashed,color=brown]
   (m-1-9) edge (m-2-8)
   (m-3-11) edge (m-2-10)
   (m-5-9) edge (m-4-10)
   (m-3-7) edge (m-4-8);

 \path[-,dashed,color=red]
   (m-1-9) edge (m-2-10)
   (m-3-11) edge (m-4-10)
   (m-5-9) edge (m-4-8)
   (m-3-7) edge (m-2-8);

 \path[-,dashed,color=blue]
   (m-1-9) edge (m-4-10)
   (m-3-11) edge (m-4-8)
   (m-5-9) edge (m-2-8)
   (m-3-7) edge (m-2-10);
   
    \path[-,dashed,color=green]
   (m-1-9) edge (m-4-8)
   (m-3-11) edge (m-2-8)
   (m-5-9) edge (m-2-10)
   (m-3-7) edge (m-4-10);
   
    \path[-,dashed,color=brown]
   (m-1-9) edge (m-2-8)
   (m-3-11) edge (m-2-10)
   (m-5-9) edge (m-4-10)
   (m-3-7) edge (m-4-8);
   
\end{tikzpicture} 

\caption{}\label{figure2}
\end{figure}

\textbf{STEP 2)}
Choose $s_1\in\langle s \rangle$.\\
Let $\delta_1:=\mathcal{P}(\alpha,\alpha_1,\alpha_2,\alpha_3,\alpha_4)_s(i,j)$,
 $\delta_2:=\mathcal{P}(\alpha,\alpha_1,\alpha_2,\alpha_3,\alpha_4)_s(j,-i)$,\\
  $\delta_3:=\mathcal{P}(\alpha,\alpha_1,\alpha_2,\alpha_3,\alpha_4)_s$ $(-i,-j)$,
   $\delta_4:=\mathcal{P}(\alpha,\alpha_1,\alpha_2,\alpha_3,\alpha_4)_s(-j,i)$ $\in \alpha s_1$
    for some integers $i,j$.
\\

We get the relations on $\alpha s \times \alpha s_1$ in the right figure in Figure \ref{figure2} by the properties of an $s$-plane where $s \cdot s_1 = r_1+r_2+r_3+r_4$ where $r_i\in S$ for $i=1,2,3,4$ and the relations $r_1,r_2,r_3$ and $r_4$ are represented by the red, blue, green and brown dashed lines as follows: 

\[
\begin{cases}
r_1:=r(\alpha _1,\delta_1)=r(\alpha _2,\delta_2)=r(\alpha _3,\delta_3)=r(\alpha _4,\delta_4),\\
r_2:=r(\alpha _1,\delta_2)=r(\alpha _2,\delta_3)=r(\alpha _3,\delta_4)=r(\alpha _4,\delta_1),\\
r_3:=r(\alpha _1,\delta_3)=r(\alpha _2,\delta_4)=r(\alpha _3,\delta_1)=r(\alpha _4,\delta_2),\\
r_4:=r(\alpha _1,\delta_4)=r(\alpha _2,\delta_1)=r(\alpha _3,\delta_2)=r(\alpha _4,\delta_3).
\end{cases}
\]

Then,

$X_1:=\{\mathrm{r}(\beta_1,\delta_1), \mathrm{r}(\beta_2,\delta_2), \mathrm{r}(\beta_3,\delta_3),\mathrm{r}(\beta_4,\delta_4)\}$ $\subset a_1r_1 \cap a_2r_2 \cap a_3r_3 \cap a_4r_4 $,

$X_2:=\{\mathrm{r}(\beta_1,\delta_2), \mathrm{r}(\beta_2,\delta_3), \mathrm{r}(\beta_3,\delta_4),\mathrm{r}(\beta_4,\delta_1)\}$ $\subset a_1r_2 \cap a_2r_3 \cap a_3r_4 \cap a_4r_1 $,

$X_3:=\{\mathrm{r}(\beta_1,\delta_3), \mathrm{r}(\beta_2,\delta_4), \mathrm{r}(\beta_3,\delta_1),\mathrm{r}(\beta_4,\delta_2)\}$  $\subset a_1r_3 \cap a_2r_4 \cap a_3r_1 \cap a_4r_2$,

$X_4:=\{\mathrm{r}(\beta_1,\delta_4), \mathrm{r}(\beta_2,\delta_1), \mathrm{r}(\beta_3,\delta_2),\mathrm{r}(\beta_4,\delta_3)\}$ $\subset a_1r_4 \cap a_2r_1 \cap a_3r_2 \cap a_4r_3 $.\\

$Claim)$ For each $i$, $|X_i|=1$ and $X_i$'s are mutually disjoint.

Suppose that there exist $X_i, X_j$ ($i\not = j$) such that $X_i \cap X_j \not = \emptyset$.

Choose $a \in X_i \cap X_j$, then $a \in a_k r_{k_1} \cap a_k r_{k_2}$ with $r_{k_1}\not = r_{k_2}$ for some $k,$ $ k_1$, $k_2$ .(Note that $r_1\not = r_3$ and $r_2\not =r_4$). \\
So, we get $0 < \langle a_k \cdot r_{k_1} ,a_k \cdot r_{k_2} \rangle$. Then, $a_k\in r_{k_1}r_{k_2}\subset \langle s \rangle$ but $a_k \not \in \langle s \rangle$ since $s \cdot t = a_1 + a_2 + a_3 +a_4$. So, $X_i \cap X_j = \emptyset.$

Since $s_i\in \langle s\rangle$ and $|s_1t|=4$, the claim holds.\\   
Cleary, when $s_1\in \langle t \rangle$, we can get the same result on $\alpha s \times\alpha s_1$.


\begin{remark}\label{remark3.9}
For $uv\in S$ with $|uv|=4$, it we have the structure of $\alpha u \times \alpha v$ like the first figure in Figure \ref{figure2}, then it is unique up to rotations on the ordered elements of $\alpha u$ and $\alpha v$, respectively.
\end{remark}

\textbf{STPE 3)} For $s_1\in \langle s\rangle$, $t_1\in \langle t\rangle$, we have to show that the relations on  \\
$\{\alpha_{i,j},\alpha_{-j_1,i_1},\alpha_{-i_1,-j_1,}\alpha_{j_1,-i_1}\}$ $\times$ $\{\beta_{i_2,j_2},\beta_{-j_2,i_2},\beta_{-i_2,-j_2},\beta_{j_2,-i_2}\}$ are invariant under the map $\sigma$ where $\alpha_{i_1,j_1}:=\mathcal{P}(\alpha, \alpha_1, \alpha_2, \alpha_3, \alpha_4)_s(i_1,j_1) \in \alpha s_1$, $\beta_{i_2,j_2}:=\mathcal{P}(\alpha, \beta_1,\beta_2,$ $\beta_3, \beta_4)_t$$(i_2,j_2) \in \alpha t_1$ for all $i_1,i_2,j_1,j_2\in \mathbb{Z}$.

If $t\in S_2$, then it's obvious. Otherwise, for $t_1\in \langle t\rangle$, when $t_1\in \langle s \rangle $, we get the desired one by Remark \ref{remark3.9}. When $t_1 \not \in \langle s \rangle $, we get the conclusion by STEP1), STEP2).  

\end{proof}


\begin{theorem}\label{theorem3.16}
Let $s\in S_3$, $t,r\in S^{\#}$ with  $s\wr t$, $s\wr r$ and $t\wr r$. Then, $t \sim_\alpha r$.

\end{theorem}

\begin{proof}

Note that we have $s \sim_\alpha t$ and $s \sim_\alpha r$.

Let $\alpha_i\in \alpha s$, $\beta_i\in \alpha t$ and $\delta_i\in \alpha r$ for $i=1,2,3,4$ with 
$\mathrm{r}(\alpha_1,\alpha_3)=\mathrm{r}(\alpha_2,\alpha_4)=s^{\psi}$, 
$\mathrm{r}(\beta_1,\beta_3)=\mathrm{r}(\beta_2,\beta_4)=t^{\psi}$ and 
$\mathrm{r}(\delta_1,\delta_3)=\mathrm{r}(\delta_2,\delta_4)=r^{\psi}$.

Let the relations on $\alpha s \times \alpha t$, $\alpha s \times \alpha r$ be given by Figure \ref{figure3} where $s
\cdot t=a_1+a_2+a_3+a_4$,  $s\cdot r=b_1+b_2+b_3+b_4$ for some distinct $a_i$ $\not \in \langle s\rangle$  $\cup \langle t\rangle$, $b_i$ $\not \in \langle s\rangle$  $\cup \langle r\rangle$ for $i=1,2,3,4$ and the relations $a_1,a_2,a_3$ and $a_4$(resp, $b_1,b_2,b_3$ and $b_4 $) are represented by the red, blue, green and brown (resp, dashed) lines in th first(resp, second) figure in Figure \ref{figure3}.

\begin{figure}[h]

\begin{tikzpicture}
  \matrix (m) [matrix of math nodes,row sep=0.3cm,column sep=0.3cm]
  { &  & \alpha_1 & & & & &  & \alpha_1 & &  && &  & \beta_1 & & \\
 	 & \beta_4 & & \beta_1& & &  & \delta_4 & & \delta_1&  	&& & \delta_4 & & \delta_1&   \\
 	\alpha_4 & & \beta &   & \alpha_2 & & \alpha_4 & & \beta &   & \alpha_2&& & & \beta &   &  \\
	 & \beta_3 &   & \beta_2  & & &  & \delta_3 &   & \delta_2  & 	 &&& \delta_3 &   & \delta_2  &  \\
	  & & \alpha_3 &   & & &  & & \alpha_3 &   &  && & & \beta_3 &   & \\};

  \path[dashed,-]
    (m-3-3) edge (m-1-3)
    (m-3-3) edge (m-5-3)
    (m-3-3) edge (m-3-1)
    (m-3-3) edge (m-3-5)
    ;

  \path[-]
    (m-3-3) edge (m-2-2)
    (m-3-3) edge (m-2-4)
    (m-3-3) edge (m-4-2)
    (m-3-3) edge (m-4-4)
    ;

 \path[-,color=red]
   (m-1-3) edge (m-2-4)
   (m-3-5) edge (m-4-4)
   (m-5-3) edge (m-4-2)
   (m-3-1) edge (m-2-2);

 \path[-,color=blue]
   (m-1-3) edge (m-4-4)
   (m-3-5) edge (m-4-2)
   (m-5-3) edge (m-2-2)
   (m-3-1) edge (m-2-4);
   
    \path[-,color=green]
   (m-1-3) edge (m-4-2)
   (m-3-5) edge (m-2-2)
   (m-5-3) edge (m-2-4)
   (m-3-1) edge (m-4-4);
   
    \path[-,color=brown]
   (m-1-3) edge (m-2-2)
   (m-3-5) edge (m-2-4)
   (m-5-3) edge (m-4-4)
   (m-3-1) edge (m-4-2);

     \path[dashed,-]
    (m-3-9) edge (m-1-9)
    (m-3-9) edge (m-5-9)
    (m-3-9) edge (m-3-7)
    (m-3-9) edge (m-3-11)
    ;

  \path[-,dotted]
    (m-3-9) edge (m-2-8)
    (m-3-9) edge (m-2-10)
    (m-3-9) edge (m-4-8)
    (m-3-9) edge (m-4-10)
    ;

  \path[-,dotted]
    (m-3-9) edge (m-2-8)
    (m-3-9) edge (m-2-10)
    (m-3-9) edge (m-4-8)
    (m-3-9) edge (m-4-10)
    ;

 \path[-,dashed,color=red]
   (m-1-9) edge (m-2-10)
   (m-3-11) edge (m-4-10)
   (m-5-9) edge (m-4-8)
   (m-3-7) edge (m-2-8);

 \path[-,dashed,color=blue]
   (m-1-9) edge (m-4-10)
   (m-3-11) edge (m-4-8)
   (m-5-9) edge (m-2-8)
   (m-3-7) edge (m-2-10);
   
    \path[-,dashed,color=green]
   (m-1-9) edge (m-4-8)
   (m-3-11) edge (m-2-8)
   (m-5-9) edge (m-2-10)
   (m-3-7) edge (m-4-10);
   
    \path[-,dashed,color=brown]
   (m-1-9) edge (m-2-8)
   (m-3-11) edge (m-2-10)
   (m-5-9) edge (m-4-10)
   (m-3-7) edge (m-4-8);
    \path[-,dashed,color=red]
   (m-1-9) edge (m-2-10)
   (m-3-11) edge (m-4-10)
   (m-5-9) edge (m-4-8)
   (m-3-7) edge (m-2-8);

 \path[-,dashed,color=blue]
   (m-1-9) edge (m-4-10)
   (m-3-11) edge (m-4-8)
   (m-5-9) edge (m-2-8)
   (m-3-7) edge (m-2-10);
   
    \path[-,dashed,color=green]
   (m-1-9) edge (m-4-8)
   (m-3-11) edge (m-2-8)
   (m-5-9) edge (m-2-10)
   (m-3-7) edge (m-4-10);
   
    \path[-,dashed,color=brown]
   (m-1-9) edge (m-2-8)
   (m-3-11) edge (m-2-10)
   (m-5-9) edge (m-4-10)
   (m-3-7) edge (m-4-8);

   \path[-]
    (m-3-15) edge (m-1-15)
    (m-3-15) edge (m-5-15)
    ;

  \path[dotted,-]
    (m-3-15) edge (m-2-14)
    (m-3-15) edge (m-2-16)
    (m-3-15) edge (m-4-14)
    (m-3-15) edge (m-4-16)
    ;

  \path[dotted,-]
    (m-3-15) edge (m-2-14)
    (m-3-15) edge (m-2-16)
    (m-3-15) edge (m-4-14)
    (m-3-15) edge (m-4-16)
    ;

 \path[-,dotted,color=red]
   (m-1-15) edge (m-2-16)
   (m-5-15) edge (m-4-14)
   ;

 \path[-,dotted,color=blue]
   (m-1-15) edge (m-4-16)
   (m-5-15) edge (m-2-14)
   ;
   
    \path[-,dotted,color=green]
   (m-1-15) edge (m-4-14)
   (m-5-15) edge (m-2-16)
   ;
   
    \path[-,dotted,color=brown]
   (m-1-15) edge (m-2-14)
   (m-5-15) edge (m-4-16)
  ;

 \path[-,dotted,color=red]
   (m-1-15) edge (m-2-16)
   (m-5-15) edge (m-4-14)
   ;

 \path[-,dotted,color=blue]
   (m-1-15) edge (m-4-16)
   (m-5-15) edge (m-2-14)
   ;
   
    \path[-,dotted,color=green]
   (m-1-15) edge (m-4-14)
   (m-5-15) edge (m-2-16)
   ;
   
    \path[-,dotted,color=brown]
   (m-1-15) edge (m-2-14)
   (m-5-15) edge (m-4-16)
  ;

\end{tikzpicture} 

\caption{}\label{figure3}
\end{figure}

$X_1:=\{\mathrm{r}(\beta_1,\delta_1), \mathrm{r}(\beta_2,\delta_2), \mathrm{r}(\beta_3,\delta_3),\mathrm{r}(\beta_4,\delta_4)\}$ $\subset a_1b_1 \cap a_2b_2 \cap a_3b_3 \cap a_4b_4 \cap tr$,

$X_2:=\{\mathrm{r}(\beta_1,\delta_2), \mathrm{r}(\beta_2,\delta_3), \mathrm{r}(\beta_3,\delta_4),\mathrm{r}(\beta_4,\delta_1)\}$ $\subset a_1b_2 \cap a_2b_3 \cap a_3b_4 \cap a_4b_1 \cap tr$,

$X_3:=\{\mathrm{r}(\beta_1,\delta_3), \mathrm{r}(\beta_2,\delta_4), \mathrm{r}(\beta_3,\delta_1),\mathrm{r}(\beta_4,\delta_2)\}$  $\subset a_1b_3 \cap a_2b_4 \cap a_3b_1 \cap a_4b_2\cap tr$,

$X_4:=\{\mathrm{r}(\beta_1,\delta_4), \mathrm{r}(\beta_2,\delta_1), \mathrm{r}(\beta_3,\delta_2),\mathrm{r}(\beta_4,\delta_3)\}$ $\subset a_1b_4 \cap a_2b_1 \cap a_3b_2 \cap a_4b_3 \cap tr$.

Suppose that for $t$ and $r$, we can't have the structure in the conclusion of STEP1) in Theorem \ref{theorem3.10}.$\cdots (*)$\\
$Claim)$ There always exist
 $i,j,k,l$ such that $c_i, c_j \in X_k \cap X_l$, $i\not = j, k \not = l$ where $t\cdot r=c_1+c_2+c_3+c_4$ for some distinct $c_i$'s of $S$.

In the case  of the form in the third figure in Figure \ref{figure3}  where the relations $c_1,c_2,c_3$ and $c_4$ on $\alpha t\times \alpha r$ are represented by the red, blue, green and brown dotted lines, then $\mathrm{r}(\beta_2,\delta_4)=c_1,$ $\mathrm{r}(\beta_2,\delta_2)=c_3$ $\in X_1 \cap X_3$ or  $\mathrm{r}(\beta_2,\delta_1)=c_2,$ $\mathrm{r}(\beta_2,\delta_3)=c_4$ $\in X_2 \cap X_4$ by $(*)$.

Otherwise, let $x=\mathrm{r}(\beta_1,\delta_i), y=\mathrm{r}(\beta_1,\delta_j), z=\mathrm{r}(\beta_1,\delta_k)$ and $w=\mathrm{r}(\beta_1,\delta_l)$ be contained in $X_{i_x}$, $X_{i_y}$, $X_{i_z}$ and $X_{i_w}$, respectively for $\{i_x,i_y,i_z,i_w\}=\{1,2,3,4\}$ and $\mathrm{r}(\delta_i,\delta_k)=\mathrm{r}(\delta_l,\delta_j)=t^{\psi}$ for $i+k \equiv 0, j+l \equiv 0$ $(mod 2)$ where $\{x, y, z, w\} = tr$ and we may assume that $\mathrm{r}(\beta_3,\delta_j)=x \in X_{i_w}. $ Suppose the claim doesn't hold. Then, $w\not \in X_{i_x}$. Note $\mathrm{r}(\beta_3,\delta_k) \not =z$, so $\mathrm{r}(\beta_3,\delta_k)=y \in X_{i_x}$. $\mathrm{r}(\beta_3,\delta_l)$ cannot be $x,y,w$. So, $\mathrm{r}(\beta_3, \delta_l)=z\in X_{i_y}$. Then,  $\mathrm{r}(\beta_3,\delta_i)=w \in X_{i_z}$. But there is $\beta_2\in \beta t$ such that  $\mathrm{r}(\beta_2,\delta_j)\in X_{i_x}$. Since $w\not \in X_{i_x}$,  $\mathrm{r}(\beta_1,\delta_j)=y$ and  $\mathrm{r}(\beta_3,\delta_j)=x$,  $\mathrm{r}(\beta_2,\delta_j)=z$ and it is contained in $X_{i_x}$. But $y,z \in X_{i_x}, X_{i_y}$. 
Thus, our claim must hold.

So there exists $c_i, c_j, b_{k_{1}}, b_{k_{2}}$ such that $c_i, c_j \in a_kb_{k_{1}} \cap a_kb_{k_{2}} $ for $k=1,2,3,4$ by the claim. Then, $\langle c_i\cdot a_k, c_j \cdot a_k \rangle \ge 8$ for $k=1,2,3,4.$ If there is $k$ such that $c_{c_i c_j}^{a_k^{\psi}}=2$, then $c_i \cdot c_j =t^{\varphi^m} +2 a_k^{\psi} + d$ for some $d \in S, m,n \in \{1,2\}$ because $a_l \not \in \langle t \rangle $ for $l=1,2,3,4.$
 But for all $a_l \in st - \{ a_k \}$, $c_{c_ic_j}^{a_l^{\psi}} \not= 2$. So, $c_{c_ic_j}^{a_l^{\varphi}} \ge1$ for all $a_l \in st -\{ a_k \}$. But $a_l$'s are distinct, $\sum_{w \in S^{\#}}c_{c_ic_j}^w > 4$. So, for all $k$, $c_{c_ic_j}^{a_k^{\varphi}}>0$ and $t^{\varphi^m} \in c_ic_j$ for some $m,n\in \{ 1,2\}$. Then, $\sum_{w \in S^{\#}}c_{c_ic_j}^w > 4$ since $a_i$'s are distinct.
 
 From STEP2) and STEP2) in the proof of Theorem~\ref{theorem3.10}, we get the conclusion.
 
\end{proof}

Let $(\Omega,S)$ be a 4-equvalenced association scheme with $S_3\not =\emptyset$.
Choose $s_1\in S_3$ such that $|\langle s_1\rangle|$ is maximal for $s_1\in S_3$ and choose $s_i\in S^{\#}$ such that $|\langle s_i\rangle|$ is maximal among $s_i\not \in \cup _{j=1}^{i-1}\langle s_j \rangle$ for $i=2,3, \cdots$. First, fix an $s_1$-plane with having $\alpha$ as a (0,0)-th image. Then, for each $i=2, 3,\cdots$, we get an $s_i$-plane with having $\alpha$ as a (0,0)-th image which makes $s_1\sim _\alpha s_i$ by Theorem ~\ref{theorem3.10}. Since the choice of $s_i$'s, $s_i\wr s_j$ if $i\not = j$,  we get the following theorem by  Theorem \ref{theorem3.16}.

It gives the following theorem.

\begin{theorem}
Let $(\Omega,S)$ be a 4-equivalenced association scheme with $S_3\not =\emptyset$. Then, for any $\alpha \in \Omega$, there is an automorphism $\sigma_\alpha$ of $(\Omega,S)$ with order 4 such that the set of orbitals of $\langle \sigma_\alpha \rangle$ is $\{\alpha s$ $|$ $s\in S\}$.

\end{theorem}

Note that for a 4-equivalenced association scheme $(\Omega,S)$ with $|S| = 2$ or $3$, $Aut(\Omega,S)$ is not Frobenius but  $(\Omega,S)$ is Frobenius. Refer \cite{HK} and \cite{T01}.

\begin{theorem}\label{corollary3.24}
When $|S| \ge 4$, if there is an element in $Aut(\Omega,S)$ which fixes two points, then it must be identity map.

\end{theorem}

\begin{proof}

Choose an element $\tau\in Aut(\Omega,S)$ which fixes two distinct points and is not identity map.
 Let $\alpha_1^{\tau} =\alpha_1,\alpha_2^\tau=\alpha_2$, $\beta^\tau=\gamma$, $\mathrm{r}(\alpha_1,\alpha_2)=s$ and $\mathrm{r}(\beta,\gamma)=r\not=1_\Omega$. If there are at least 4 fixed points, then $\sum_{t\in S^{\#}}c_{tt}^r \ge 4$. It's a contradiction. So, we only need to show that there are at least 4 fixed points.
If $|S|>4$, there is $t\in S^{\#}$ such that $|st|=4$. Then, the elements in $\alpha_2 t$ are also fixed by $\tau$, otherwise $\tau$ cannot preserve the relations on $\alpha_2 s \times  \alpha_2 t$. If $|S|=$ 4, then $S_2=\emptyset$. Since $c_{ss}^{s^\psi}=1$ and $s\ne s^\psi$, we get the four fixed points.

\end{proof}

\begin{theorem}\cite{P15}
Let $(\Omega,S)$ be a 4-equivalenced association scheme and $S^{\#}$=$S_2$.\\Then, $(\Omega,S)$ is Frobenius and $|\Omega|=5^l$ for some $l\in \mathbb{Z}^+$.
\end{theorem}

So, we reach our main goal.

\begin{corollary}\label{corollary3.21}
Every 4-equivalenced association scheme is Frobenius.
\end{corollary}

\section{Other results on a 4-equivalenced association scheme}

See \cite{MP12} for the definitions of coherent configuration, fiber, semiregularity,  separability, isomorphism and algebraic isomorphism of a coherent configuration.\\



Let $(\Omega,S)$ be a 4-equivalenced association scheme.

Use the notations $\mathrm{r}(a,b)_\alpha$ (resp, the element $s_\alpha$ of $S_\alpha$)  for an $\alpha$-fission $(\Omega,S_\alpha)$ instead of $\mathrm{r}(a,b)$ (resp, the element $s$ of $S$) for an association scheme $(\Omega,S )$ and $p_{{u_\alpha}{v_\alpha}}^{w_\alpha}$ for an $\alpha$-fission $(\Omega,S_\alpha)$ instead of $c_{uv}^{w}$ for an association scheme $(\Omega,S )$.

\begin{remark}$\empty$

\begin{enumerate} 

\item For any $s_\alpha \in S_{\alpha}$, there exists a unique pair of fibers $\Gamma, \Delta$ of an $\alpha$-fission $(\Omega,S_\alpha)$ such that $s_\alpha \subset \Delta \times \Gamma$.
\item Since $\{\alpha\}$ is also a fiber of an $\alpha$-fission $(\Omega,S_{\alpha})$, for any fiber $\Gamma$ in $S_{\alpha}$, there exists $u\in S$ such that $\Gamma \subset \alpha u$.

\end{enumerate}

\end{remark}

\begin{lemma}\label{lemma3.4}
When $|S|\ge 5$, for any $u\in S^{\#}$ there is $v\in S^{\#}$ with $|uv|=4$.
\end{lemma}

\begin{proof}
Since $|S|\ge 5$, we can choose $v \not \in$ $\{u,u^{\varphi}, u^{\varphi^{-1}}\}$. We have it by Lemma~\ref{lemma2.6}.

\end{proof}


Choose a point $\alpha\in \Omega$. Then, we get the following theorem.

\begin{theorem}\label{theorem3.6} 
A fission $(\Omega,S_{\alpha})$ of an association scheme $(\Omega,S)$ is semiregular on $\Omega$ $\backslash$ $\{\alpha\}$ if $|S|\ge 3$.

\end{theorem}

\begin{proof}
When $|S|=$ 3 or 4, it's semiregular on $(\Omega,S)$ by some calculations.\\
When $|S|\ge 5$, suppose that there is $t_\alpha \in S_\alpha \cap (\alpha u \times \alpha v)$ such that $|\alpha_1 t_\alpha| \ge 2$ for some $\alpha_1\in \alpha u$.
Let $t_\alpha:=\mathrm{r}(\alpha_1,\beta_1)=\mathrm{r}(\alpha_1,\beta_2)$ where $\beta_1,\beta_2\in \alpha v$  and $v_\alpha:=\mathrm{r}(\alpha,\beta_1)_\alpha=\mathrm{r}(\alpha,\beta_2)_\alpha$. 
By Lemma~\ref{lemma3.4}, there exists $w \in S^{\#}$ such that $u\cdot w=r_1+r_2+r_3+r_4$ and $|uw|=4$.
Let ${r_i}_\alpha  \subset r_i\cap(\alpha u\times\alpha w)$ and $\delta_j\in\alpha w$ such that $\mathrm{r}(\alpha_1,\delta_j)_\alpha={r_j}_\alpha$ for $i,j=1,2,3,4$.
\begin{center}
\begin{tabular}{c|cccc|cccc}
                &   & $ \beta_1$     & $ \beta_2 $  & $ $ &  $\delta_1$ & $\delta_2$ & $\delta_3$ & $\delta_4$\\ 
 \hline
$\alpha$  &     &  $v_\alpha$ &$v_\alpha$&                          &   $w_\alpha$&&& \\ 
\hline
$\alpha_1$              &     &   $t_\alpha$    &   $t_\alpha$  & & ${r_1}_\alpha$     &   ${r_2}_\alpha$  &   ${r_3}_\alpha$   &${r_4}_\alpha$ \\

\end{tabular}
\end{center}

$case$ $1)$ When $u=v$, let $s_\alpha:=\mathrm{r}(\beta_1,\delta_1)_\alpha$ in $S_{\alpha}$. Then, $\mathrm{r}(\beta_2,\delta_1)_\alpha={s}_\alpha$ since $\mathrm{r}(\alpha_1,\beta_2)_\alpha=t_\alpha$ and $\delta_1$ is a unique point with $\mathrm{r}(\alpha_1,\delta_1)_\alpha={r_1}_\alpha$.

So, $p_{v_\alpha s_\alpha}^{w_\alpha}\ge2$ where $w_\alpha:=\mathrm{r}(\alpha, \delta_1)_\alpha \subset \{\alpha\}\times \alpha w$. Then, $c_{u s}^{w} \ge 2$ where ${s}_\alpha \subset s, s\in S$. But $|uw|=4$. This is a contradiction to our choice of $w$.\\

$case$ $2)$ When $u\not=v$.
Let ${s_j}_\alpha:=\mathrm{r}(\beta_1,\delta_j)_\alpha$ for $j=1,2,3,4$. Then, $\mathrm{r}(\beta_2,\delta_j)_\alpha={s_j}_\alpha$ for $j=1,2,3,4$ since $\delta_j$ is a unique point with $\mathrm{r}(\alpha_1,\delta_j)_\alpha={r_j}_\alpha$. 
Let $t_\alpha=\mathrm{r}(\beta_1,\beta_2)_\alpha$.
 Then, for $j=1,2,3,4$, $p_{{s_j}_\alpha {s_j}_\alpha*}^{t_\alpha}\ge1$. We get a contradiction since every 4-equivalenced association scheme is pseudocyclic and we have $\sum_{r\in S^{\#}}c_{r^{\*}r}^t\ge4$ where $t \in S $ and $t_\alpha \subset t$.

\end{proof}

\begin{theorem}
Every 4-equivalenced association scheme $(\Omega, S)$ is separable.
\end{theorem}

\begin{proof}
When $|S|=2$, it's trivial. When $|S|\ge3$, it follows from Theorem 7.2 in \cite{MP12}, Corollary \ref{corollary3.21} and Theorem \ref{theorem3.6}.

\end{proof}



\begin{definition}\cite{P13}
Let $\mathcal{X}$ be a coherent configuration. A point set $\Delta$ is called a \textit{base} of $\mathcal{X}$ if the smallest fission of $\mathcal{X}$ in which all points of $\Delta$ are fibers, is the complete coherent configuration. Here, the complete configuration is coherent configuration whose relation sets are all binary relations on $\Omega\times\Omega$.\\
And, $b(\mathcal{X})=$minimum size of $|\Delta|$ for $\Delta\subset\Omega$ such that $S_{\Delta}$-fission is a complete configuration of $\mathcal{X}$. 
\end{definition}

\begin{theorem}
Let $\mathcal{X}=(\Omega,S)$ be a 4-equivalenced association scheme with $S_2\ne \emptyset$ and $|S|\not =2.$ Then, $b(\mathcal{X})=2$.
\end{theorem}

\begin{proof}
By Theorem \ref{theorem3.6}, we know $b(\mathcal{X})\ge2$ since $S_\alpha$-fission is semiregular on $\Omega\backslash\{\alpha\}$ and has fibers $\alpha u$, $u\in S$. Let $\Delta=\{\alpha, \beta\}$ such that $\mathrm{r}(\alpha,\beta)\in S_2$. Now, consider $S_{\{\alpha,\beta\}}$-fission. In $S_\alpha$-fission, for any $t \in S^{\#}$ with $t\ne \mathrm{r}(\alpha, \beta)$, $|\mathrm{r}(\alpha,\beta)t|=4$. Since \textrm{r}$(\beta,\gamma_i)$'s are distinct, where $\gamma_i \in \alpha t, i=1,2,3,4$, for any $t\in S$ except $\mathrm{r}(\alpha,\beta)$, each point in $\alpha t$ must be a fiber in $S_{\{\alpha,\beta\}}$-fission.\\
Similarly, we can show that every element in $\alpha s$ is also fiber.
\end{proof}

From Theorem \ref{theorem2.3}, the following is obvious.
\begin{corollary}\cite{MP12}
For any 4-equivalenced association scheme $(\Omega,S)$, a pair of $(\Omega,\mathcal{B})$ with $\mathcal{B}=\{\alpha s : \alpha \in \Omega, s\in S^{\#}\}$ is a 2-$(n,4,3)-$design where $n=|\Omega|$.

\end{corollary}

\begin{remark} \cite{T01} A 4-equivalenced association scheme exists if and only if there exists degree m, order 4m Frobenius group by Boykett.T. And it's known when the Frobenius group exists in \cite{T01}.

\end{remark}



\end{document}